\documentclass[12pt]{amsart}
\pdfoutput=1
\usepackage{amsfonts}
\usepackage{mathrsfs}
\usepackage{amssymb}
\usepackage{amsmath}
\usepackage{amsrefs}
\usepackage{centernot}
\usepackage{fullpage}
\usepackage{color}

\newtheorem{theorem}{Theorem}[section]
\newtheorem{lemma}[theorem]{Lemma}
\newtheorem{proposition}[theorem]{Proposition}
\newtheorem{corollary}[theorem]{Corollary}

\theoremstyle{definition}
\newtheorem{definition}[theorem]{Definition}

\theoremstyle{remark}
\newtheorem{remark}[theorem]{Remark}

\newcommand{\B}{\mathcal B}
\renewcommand{\H}{\mathcal H}
\newcommand{\K}{\mathcal K}
\newcommand{\RR}{\mathbb R}
\newcommand{\CC}{\mathbb C}
\newcommand{\<}{\langle}
\renewcommand{\>}{\rangle}
\renewcommand{\:}{\colon}
\newcommand{\suchthat}{\,|\,}
\newcommand{\tensor}{\mathbin{\otimes}}
\newcommand{\iso}{\cong}
\newcommand{\IM}{\quad\Longrightarrow\quad}
\newcommand{\tr}{\mathrm{tr}}
\newcommand{\trr}{\mathrm{t}\tilde{\mathrm r}}
\newcommand{\id}{\mathrm{id}}

\allowdisplaybreaks

\begin{document}

\title{Characterizations of homomorphisms \\ among unital completely positive maps}
\author{Andre Kornell}
\address{Department of Mathematics and Statistics, Dalhousie University, Halifax, Nova Scotia}
\email{akornell@dal.ca}
\thanks{This work was
supported by the Air Force Office of Scientific Research under Award No. FA9550-21-1-0041.}

\begin{abstract}
We prove that a unital completely positive map between finite-dimensional $C^*$-algebras is a homomorphism if and only if it is completely entropy-nonincreasing, where the relevant notion of entropy is a variant of von Neumann entropy. This adjusted von Neumann entropy is the negative of the relative entropy with respect to the uniform state on the $C^*$-algebra, up to an additive constant. As an intermediate step, we prove that a unital completely positive map between finite-dimensional $C^*$-algebras is a homomorphism if and only if its adjusted Choi operator is a projection. Both equivalences generalize familiar facts about stochastic maps between finite sets.
\end{abstract}

\maketitle

\section{Introduction}

Let $A$ be a finite-dimensional $C^*$-algebra. Up to isomorphism, $A = M_{n_1}(\CC) \oplus \cdots \oplus M_{n_\ell}(\CC)$ for some positive integers $n_1, \ldots, n_\ell$. We define $\tr\: A \to \CC$ by
$$
\tr(a_1 \oplus \ldots \oplus a_\ell) = \tr(a_1) + \cdots + \tr(a_\ell).
$$
We define the \emph{entropy} of a state $\mu\: A \to \CC$ by $S(\mu) = - \mu(\log d)$, where $d \in A$ is the unique operator such that $\mu(a) = \tr(da)$. This definition generalizes the von Neumann entropy of a mixed state in the obvious way \cite{VonNeumann}*{section~V.2}.
We work primarily with a different but closely related notion of entropy. We define the \emph{adjusted entropy} of $\mu$ by $\tilde S(\mu) = -\mu(\log d) + \mu(\log \zeta_A)$, where $\zeta_A \in A$ is defined by
$
\zeta_A = n_1 1_{n_1} \oplus \cdots \oplus n_\ell 1_{n_\ell}.
$

The main result of the present paper is the following theorem.

\begin{theorem}[also Corollary~\ref{Q}]\label{A}
Let $A$ and $B$ be finite-dimensional $C^*$-algebras, and let $\varphi\: A \to B$ be a unital completely positive map. The following are equivalent:
\begin{enumerate}
\item $\varphi$ is a homomorphism,
\item for all positive integers $k$ and all states $\mu$ on $M_k(B) = B \tensor M_k(\CC)$,
$$
\tilde S(\mu \circ (\varphi \tensor \id)) \leq \tilde S(\mu).
$$
\end{enumerate}
\end{theorem}
\noindent We might paraphrase this equivalence by saying that ``a quantum operation between finite physical systems is dual to a homomorphism iff it is completely entropy-nonincreasing.'' Here, we have used the term ``quantum operation'' in the sense of Kraus \cite{Kraus}, and we have used the term ``entropy-nonincreasing'' to mean that the \emph{adjusted} entropy does not increase. This may be the first characterization of unital $*$-homomorphisms in terms of physical notions.

We close this introduction with six remarks.

\begin{remark}
Theorem~\ref{A} is certainly \emph{not} true with the unadjusted entropy $S$ in place of the adjusted entropy $\tilde S$. Let $A = \CC^2$, let $B = M_2(\CC)$, and let $\varphi(a_1, a_2) = \left(\begin{smallmatrix} a_1 & 0 \\ 0 & a_2\end{smallmatrix}\right)$. Let $k = 1$, and let $\mu\left(\begin{smallmatrix}  b_{11} & b_{12} \\ b_{21} & b_{22}\end{smallmatrix}\right) =  \frac 1 2(b_{11} + b_{12} + b_{21} + b_{22})$. The state $\mu$ is pure, and hence, $S(\mu) = 0$. However, the state $\mu \circ \varphi$ is uniform, and hence, $S(\mu \circ \varphi) = \log 2$. Thus, $S(\mu \circ \varphi) \not \leq S(\varphi)$. Nevertheless, $\varphi$ is a homomorphism.
\end{remark}

\begin{remark}\label{B}
Theorem~\ref{A} has no obvious analogue for infinite-dimensional $C^*$-algebras or for infinite-dimensional von Neumann algebras. Let $\H$ be a Hilbert space. If $\H$ is finite-dimensional, then the adjusted entropy of a pure state on $\K(\H) = \B(\H)$ is $\log(\dim \H)$. Thus, if $\H$ is infinite-dimensional, then the adjusted entropy of a pure state on $\K(\H)$ or, equivalently, of a pure normal state on $\B(\H)$ appears to be infinite.
\end{remark}

\begin{remark}
The adjusted entropy of a state is a well motivated notion without Theorem~\ref{A}. It is the notion of entropy that we obtain by working with the adjusted trace $\trr(a) = \tr(a\zeta_A)$ instead of the unadjusted trace $\tr(a)$; see Definitions~\ref{K} and \ref{M}. This adjusted trace is the canonical trace on a finite-dimensional $C^*$-algebra in a couple of contexts.

In noncommutative geometry, finite-dimensional $C^*$-algebras are viewed as generalizing finite sets. The adjusted trace is then the corresponding generalization of counting measure. For illustration, counting measure is the Haar measure on any finite group, and similarly, the adjusted trace is the Haar state on any finite compact quantum group \cite{Woronowicz}*{appendix A.2}.

In categorical quantum mechanics, finite-dimensional $C^*$-algebras are viewed as special unitary dagger Frobenius algebras in the dagger compact closed category of finite-dimensional Hilbert spaces. In these dagger Frobenius algebras, the multiplication and comultiplication are adjoint for the inner-product that is obtained from the adjusted trace in the obvious way \cite{Vicary}*{Theorem~4.6}.
\end{remark}

\begin{remark}
The adjusted entropy $\tilde S (\mu)$ is closely related to the relative entropy $S(\mu || \nu)$ \cite{Umegaki}*{section~4}, where $\nu\: A \to \CC$ is the uniform state on $A$. Explicitly, $\nu(a) := \trr(a)/\dim A$ for all $a \in A$. We then have that $\tilde S(\mu) = \log(\dim A) - S(\mu||\nu)$.
\end{remark}

\begin{remark}
Theorem~\ref{A} suggests that the adjusted entropy $\tilde S$ is the right notion of entropy in the mathematical setting of finite-dimensional $C^*$-algebras. Is the adjusted entropy $\tilde S$ also the right notion of entropy in the physical setting of finite physical systems?

This question has empirical content. The term $\mu(\zeta_A)$ in the expression for $\tilde S(\mu)$ suggests that each superselection sector of a finite physical system carries intrinsic physical entropy according to its dimension. A state that maximizes $\tilde S$ assigns greater weight to superselection sectors of large dimension than a state that maximizes $S$.
\end{remark}

\begin{remark}
Theorem~\ref{A} may be restated in terms of trace-preserving completely positive maps, which are dual to unital completely positive maps. This perspective is more common in quantum information theory.

Let $A$ and $B$ be finite-dimensional $C^*$-algebras, and let $\psi\: A \to B$ be a trace-preserving completely positive map. Theorem~\ref{A} implies that $\psi^\dagger\: B \to A$ is a homomorphism iff $\tilde S((\psi \tensor \id)(d)) \leq \tilde S(d)$ for all positive integers $k$ and all density operators $d \in M_k(A)$. Here, we define $\psi^\dagger$ by $\tr(a \psi^\dagger(b)) = \tr(\psi(a)b)$ and $\tilde S(d) = - \tr(d \log d) + \tr(d \log (\zeta_A \tensor 1_k)) + \log(k)$.
\end{remark}

We refrain from choosing a specific base for the logarithm. Thus, the notation $\log(x)$ refers to the logarithm of $x$ with respect to an arbitrary fixed base. For convenience, we define a $C^*$-algebra to be nonzero. Thus, we always have that $1 \neq 0$.

\section{Adjusted Choi matrix}

Let $n_1, \ldots, n_\ell$, and $m$ be positive integers. Let $\varphi\: M_{n_1}(\CC) \oplus \ldots \oplus M_{n_\ell}(\CC) \to M_m(\CC)$
be a linear map.
Let $e$ be the projection matrix
$$
e = \sum_{k = 1}^\ell \sum_{i,j = 1}^{n_k} \frac 1 n_k e_{ijk} \tensor e_{ijk},
$$
where $e_{ijk}$ is the matrix unit
$0_{n_1} \oplus \cdots \oplus 0_{n_{k-1}} \oplus e_{ij} \oplus 0_{n_{k+1}} \oplus \cdots \oplus 0_{n_\ell}$.

\begin{definition}
The \emph{Choi matrix} and the \emph{adjusted Choi matrix} of $\varphi$ are
$$
c_\varphi = \sum_{k = 1}^\ell \sum_{i,j = 1}^{n_k} \varphi(e_{ijk}) \tensor e_{ijk},
\qquad
\tilde c_ \varphi = (\varphi \tensor \id)(e) = \sum_{k = 1}^\ell \sum_{i,j = 1}^{n_k} \frac 1 n_k \varphi(e_{ijk}) \tensor e_{ijk},$$
respectively.
\end{definition}

The Choi matrix $c_\varphi$ was introduced by Choi \cite{Choi}, following de Pillis \cite{DePillis}, in the case $\ell = 1$. Choi showed that $\varphi$ is completely positive iff $c_\varphi$ is positive \cite{Choi}*{Theorem~2}. We briefly verify that this equivalence holds for the adjusted Choi matrix $\tilde c_\varphi$.

\begin{proposition}
The map $\varphi$ is completely positive iff $\tilde c_\varphi$ is positive. 
\end{proposition}

\begin{proof}
The matrix $e$ is positive because it is a projection. If $\varphi$ is completely positive, then the adjusted Choi matrix $\tilde c_\varphi = (\varphi \tensor \id)(e)$ is immediately positive too.

Conversely, assume that $\tilde c_\varphi$ is positive. For each $1 \leq k \leq \ell$, let $$\pi_k\: M_{n_k}(\CC) \to M_{n_1}(\CC) \oplus \ldots \oplus M_{n_\ell}(\CC)$$ be the obvious inclusion map. Thus, $\pi_k(e_{ij}) = e_{ijk}$. The matrix
$$
\sum_{i,j = 1}^{n_k} \frac 1 n_k \varphi(\pi_k(e_{ij})) \tensor e_{ij}
$$
is positive because it is a block of $\tilde c_\varphi$. By Choi's theorem \cite{Choi}*{Theorem~2}, we conclude that $\varphi \circ \pi_k$ is completely positive for all $k$ and, more generally, that $\varphi$ is completely positive.
\end{proof}

\begin{theorem}\label{C}
Assume that $\varphi\:M_{n_1}(\CC) \oplus \ldots \oplus M_{n_\ell}(\CC) \to M_m(\CC)$ is completely positive and completely contractive \cite{EffrosRuan}*{section~2.2}. Then, $\varphi$ is a homomorphism iff $\tilde c_\varphi$ is a projection.
\end{theorem}

\begin{proof}

If $\varphi$ is a homomorphism, then $\tilde c_\varphi = (\varphi \tensor \id)(e)$ is a projection because $e$ is a projection.

Conversely, assume that $\tilde c_\varphi$ is a projection. Then, $\varphi \tensor \id$ is completely positive, completely contractive, and satisfies $(\varphi \tensor \id)(e)^* (\varphi \tensor \id)(e) = (\varphi \tensor \id)(e^*e)$. It follows that $(\varphi \tensor \id)(ae) = (\varphi \tensor \id) (a) (\varphi \tensor \id)(e)$ for all $a \in (M_{n_1}(\CC) \oplus \ldots \oplus M_{n_\ell}(\CC))^{\tensor 2}$ \cite{EffrosRuan}*{Corollary~5.2.2}. Taking $a = e_{rst} \tensor e_{uvw}$ for some $1 \leq r, s \leq n_t$ and $1 \leq u, v \leq n_w$ with $1 \leq t, w \leq \ell$, we find that
$$\sum_{k = 1}^\ell \sum_{i,j = 1}^{n_k} \frac 1 n_k \varphi (e_{rst} e_{ijk}) \tensor e_{uvw} e_{ijk} = \sum_{k = 1}^\ell \sum_{i,j = 1}^{n_k} \frac 1 n_k \varphi (e_{rst}) \varphi(e_{ijk}) \tensor e_{uvw} e_{ijk},$$
$$\sum_{i,j = 1}^{n_w} \frac 1 n_w \varphi (e_{rst} e_{ijw}) \tensor  e_{uvw} e_{ijw} = \sum_{i,j = 1}^{n_w} \frac 1 n_w \varphi (e_{rst}) \varphi(e_{ijw}) \tensor e_{uvw} e_{ijw},$$
$$\sum_{j = 1}^{n_w} \frac 1 n_w \varphi (e_{rst} e_{vjw}) \tensor e_{uvw} e_{vjw} = \sum_{j = 1}^{n_w} \frac 1 n_w \varphi (e_{rst}) \varphi(e_{vjw}) \tensor e_{uvw} e_{vjw},$$
$$\sum_{j = 1}^{n_w} \frac 1 n_w \varphi (e_{rst} e_{vjw}) \tensor e_{ujw} = \sum_{j = 1}^{n_w} \frac 1 n_w \varphi (e_{rst}) \varphi(e_{vjw}) \tensor e_{ujw},$$
$$\frac 1 n_w \varphi (e_{rst} e_{vuw}) \tensor e_{uuw} =  \frac 1 n_w \varphi (e_{rst}) \varphi(e_{vuw}) \tensor e_{uuw},$$
$$
\varphi (e_{rst} e_{vuw}) = \varphi (e_{rst}) \varphi(e_{vuw}).
$$
Since $\{e_{ijk} \suchthat 1 \leq i, j \leq n_k,\, 1 \leq k \leq \ell\}$ is a basis for $M_{n_1}(\CC) \oplus \ldots \oplus M_{n_\ell}(\CC)$, we conclude that $\varphi$ is a homomorphism.
\end{proof}

\begin{corollary}\label{D}
Let $n$ and $m$ be positive integers, and let $\varphi\: M_n(\CC) \to M_m(\CC)$ be a unital completely positive map. Then, $\varphi$ is a homomorphism iff the matrix
$$
\tilde c_\varphi = \frac 1 n \sum_{i, j = 1}^n \varphi(e_{ij})\tensor e_{ij}
$$
is a projection.
\end{corollary}

\begin{proof}
This is the case $\ell = 1$, because every unital completely positive map is completely contractive \cite{EffrosRuan}*{Corollary~5.1.2}.
\end{proof}

The Choi-Jamio{\l}kowski isomorphism, which is also known as the channel-state duality, is a well-known one-to-one correspondence between trace-preserving completely positive maps $\psi \: M_n(\CC) \to M_m(\CC)$ and certain density matrices $d \in M_m(\CC) \tensor M_n(\CC)$; see \cites{Choi, Jamiolkowski, JiangLuoFu} and \cite{qubitguide}*{section 9.9}. It associates each ``channel'' $\psi$ to the ``state'' $d = (\psi \tensor \id)(e) = \tilde c_\psi$. 

\begin{corollary}
Let $n$ and $m$ be positive integers, and let $\psi \: M_m(\CC) \to M_n(\CC)$ be a trace-preserving completely positive map. Then, the Choi-Jamio{\l}kowski isomorphism maps $\psi$ to the density matrix $\tilde c_\psi$, and $\psi^\dagger$ is a homomorphism iff $\frac m n \tilde c_\psi$ is a projection.
\end{corollary}

\begin{proof}
The Choi-Jamio{\l}kowski isomorphism maps $\psi$ to $\tilde c_\psi$ by its definition. We write $\sigma$ for the canonical $*$-isomorphism $M_n(\CC) \tensor M_m(\CC) \to M_m(\CC) \tensor M_n(\CC)$. Defining $\<a|b\> = \tr(a^*b)$, we calculate that for all $1 \leq r, s \leq m$ and $1 \leq u, v \leq n$,
\begin{align*}
\< n \tilde c_{\psi^\dagger} | e_{rs} \tensor e_{uv}\>
& =
\sum_{i,j =1}^n \< \psi^\dagger(e_{ij}) \tensor e_{ij} | e_{rs} \tensor e_{uv} \>
=
\<\psi^\dagger(e_{uv}) |e_{rs}\>
=
\<e_{uv} | \psi(e_{rs})\>
\\ & =
\sum_{i,j = 1}^m \< e_{rs} \tensor e_{uv} | e_{ij} \tensor \psi(e_{ij})\>
=
\< e_{rs} \tensor e_{uv} | m\sigma(\tilde c_\psi)\>
=
\overline{\< m\sigma(\tilde c_\psi) | e_{rs} \tensor e_{uv} \>}.
\end{align*}
Thus, $\tilde c_{\psi^\dagger}$ is the conjugate of the matrix $\sigma \big(\frac m n \tilde c_\psi \big)$. We now reason that $\psi^\dagger$ is a homomorphism iff $\tilde c_{\psi^\dagger}$ is a projection iff $\sigma \big(\frac m n \tilde c_\psi \big)$ is a projection iff $\frac m n \tilde c_\psi$ is a projection, using Corollary~\ref{D}.
\end{proof}

\section{Heisenberg picture}

Let $A$ be a finite-dimensional $C^*$-algebra.

\begin{definition}
 A \emph{trace} on $A$ is a positive linear map $\mu \: A \to \CC$ such that $\mu(ab) = \mu(ba)$. We specify two traces $\tr, \trr\: A \to \CC$ as follows:
\begin{enumerate}
\item $\tr(p) = 1$ for each minimal projection $p \in A$,
\item $\trr(p) = \dim A p$ for each projection $p \in A$.
\end{enumerate}
These two traces are related by $\trr(a) = \tr(a \zeta_A)$ for a unique \emph{dimension operator} $\zeta_A \in A$. Explicitly, if $A = M_{n_1}(\CC) \oplus \cdots \oplus M_{n_\ell}(\CC)$, then $\zeta_A = n_1 1_{n_1} \oplus \cdots \oplus n_\ell 1_{n_\ell}$.
\end{definition}

The dimension operator evidently satisfies 
\begin{enumerate}
\item $\zeta_{A \oplus B} = \zeta_A \oplus \zeta_B$,
\item $\zeta_{A \otimes B} = \zeta_A \otimes \zeta_B$,
\item $\zeta_{A^{op}} = \zeta_A$.
\end{enumerate}

Recall that $A^{op}$ is the $C^*$-algebra that has the same underlying vector space as $A$ but whose multiplication reverses the order of the multiplication in $A$. Thus, $A$ and $A^{op}$ are equal as $C^*$-algebras iff $A$ is commutative, but $A$ and $A^{op}$ are always equal as vector spaces. Furthermore, because the tensor product of $A$ and $B$ is their tensor product as vector spaces, the $C^*$-algebras $A \tensor B$, $A^{op} \tensor B$, $A \tensor B^{op}$, and $A^{op} \tensor B^{op}$ are all equal as vector spaces. Furthermore, $A^{op} \tensor B^{op}$ and $(A \tensor B)^{op}$ are equal as $C^*$-algebras.

\begin{definition}\label{E}
Let $\delta$ be the maximum projection in $A \tensor A^{op}$ that is orthogonal to $p \tensor (1-p)$ for all projections $p \in A$. We say that $\delta$ is the \emph{diagonal projection} \cite{quantumlogic}.
\end{definition}

We faithfully represent $A \tensor A^{op}$ on the Hilbert space $A$, which we equip with the inner product $(a|b) = \trr(a^* b)$.
Operators of the form $a \tensor 1$ become multiplication on the left by $a$, and operators of the form $1 \tensor a$ become multiplication on the right by $a$. This inner product on $A$ has appeared previously in the literature, e.g., in the context of finite-dimensional $C^*$-algebras as certain dagger Frobenius algebras in the dagger category of finite-dimensional Hilbert spaces \cite{Vicary}.

\begin{proposition}\label{F}
Let $A \tensor A^{op}$ be represented on $A$. The range of $\delta$ is the center of $A$.
\end{proposition}

\begin{proof}
Let $a \in A$. Assume that $a$ is in the range of $\delta$. Then, $a$ is in the kernel of $p \tensor (1-p)$ for all projections $p \in A$, and thus, $p a (1-p) = 0$ and $(1-p) a p = 0$ for all projections $p \in A$. We find that $a$ commutes with every projection in $A$, so $a$ is in the center of $A$. Conversely, if $a$ is in the center of $A$, then $p a (1-p) = 0$ for all projections $p \in A$, and thus, $a$ is in the range of $\delta$. Therefore, the range of $\delta$ is equal to the center of $A$.
\end{proof}

\begin{lemma}\label{G}
For all $a, b \in A$, we have that
$
\trr(\delta (a \tensor b)) = \trr(ab).
$
\end{lemma}

\begin{proof}
Let $A = A_1 \oplus \cdots \oplus A_\ell$ be a decomposition of $A$ into simple $C^*$-algebras. Furthermore, let $1 = p_1 + \cdots + p_\ell$ and $\zeta_A = n_1 p_1 + \ldots + n_\ell p_\ell$ be the corresponding decompositions of $1 \in A$ and $\zeta_A \in A$, respectively. The range of $\delta$ is the center of $A$ by Proposition~\ref{F}, and hence, $\{p_1, \ldots, p_n\}$ is an orthogonal basis for the range of $\delta$. Since $(p_k|p_k) = \trr(p_k)= \tr(p_k\zeta_A) = n_k \tr(p_k) = n_k^2$, the set $\{n_1^{-1}p_1, \ldots, n_\ell^{-1}p_n\}$ is an orthonormal basis for the range of $\delta$.

We now calculate that 
\begin{align*}
\trr(\delta(a \tensor b))
& =
\tr(\delta(a \tensor b) \zeta_{A \tensor A^{op}})
=
\tr(\delta(a \tensor b) (\zeta_A \tensor \zeta_A))
 \\ & =
\tr(\delta(a \tensor b) (\zeta_A \tensor \zeta_A) \delta)
=
\sum_{k = 1}^\ell ( n_k^{-1}p_k|(a \tensor b) (\zeta_A \tensor \zeta_A)n_k^{-1}p_k)
\\ & =
\sum_{k = 1}^\ell n_k^{-2} ( p_k|(a \tensor b) (\zeta_A \tensor \zeta_A) p_k)
=
\sum_{k = 1}^\ell n_k^{-2} ( p_k|(a \tensor b) (\zeta_A p_k \zeta_A))
\\ & =
\sum_{k = 1}^\ell n_k^{-2} ( p_k|(a \tensor b) n_k^2 p_k)
=
\sum_{k = 1}^\ell  ( p_k|(a \tensor b) p_k)
=
\sum_{k = 1}^\ell  ( (p_k \tensor p_k) 1|(a \tensor b) p_k)
\\ & =
\sum_{k = 1}^\ell  (  1|(a \tensor b) (p_k \tensor p_k) p_k)
= 
\sum_{k = 1}^\ell  (  1|(a \tensor b) p_k)
=
(1 | (a \tensor b) 1) = ( 1 | a b) = \trr(ab).
\end{align*}
\end{proof}

\begin{lemma}\label{H}
Assume that  $A = M_{n_1}(\CC) \oplus \cdots \oplus M_{n_\ell}(\CC)$. Let $\tau$ be the $*$-isomorphism $A^{op} \to A$ that maps each matrix to its transpose. Then,
$$
(\id \tensor \tau)(\delta) = e = \sum_{k = 1}^\ell \sum_{i,j = 1}^{n_k} \frac 1 n_k e_{ijk} \tensor e_{ijk}.
$$
\end{lemma}

\begin{proof}
The center of $A$ is $\CC 1_{n_1} \oplus \cdots \oplus \CC 1_{n_\ell}$. By Proposition~\ref{F}, this is also the range of the projection $\delta$ when $A \tensor A^{op}$ is represented on $A$. Using that $(1_{n}| 1_{n}) = \trr(1_n) = n^2$ for all positive integers $n$, we calculate that
\begin{align*}\delta(x_1 \oplus \cdots \oplus x_\ell)
& = 
\frac 1 {n_1^2} (1_{n_1} | x_1) 1_{n_1} \oplus \cdots \oplus \frac 1 {n_\ell^2} (1_{n_\ell}|x_\ell) 1_{n_\ell} 
\\ &=
\frac 1 {n_1} \tr(x_1) 1_{n_1} \oplus \cdots \oplus \frac 1 {n_\ell} \tr(x_\ell) 1_{n_\ell} 
\\ & = 
\bigoplus_{k=1}^\ell \frac 1 {n_k} \tr(x_k) 1_{n_k}
=
\bigoplus_{k=1}^\ell \sum_{i, j = 1}^{n_k} \frac 1 {n_k} (e_j^* x_k e_j) e_i e_i^*
\\ & = 
\bigoplus_{k=1}^\ell \sum_{i, j = 1}^{n_k} \frac 1 {n_k}  e_i e_j^* x_k e_j e_i^*
=
\bigoplus_{k=1}^\ell \sum_{i, j = 1}^{n_k} \frac 1 {n_k}  e_{ij} x_k e_{ji}
\\ & = 
\sum_{k=1}^\ell \sum_{i, j = 1}^{n_k} \frac 1 {n_k}  e_{ijk} (x_1 \oplus \cdots \oplus x_\ell) e_{jik}
\\ & = 
\sum_{k=1}^\ell \sum_{i, j = 1}^{n_k} \frac 1 {n_k}  (e_{ijk} \tensor e_{jik})(x_1 \oplus \cdots \oplus x_\ell)
\end{align*}
Therefore,
$$
\delta = \sum_{k=1}^\ell \sum_{i, j = 1}^{n_k} \frac 1 {n_k} e_{ijk} \tensor e_{jik}.
$$

We conclude that
$$
(\id \tensor \tau)(\delta) = \sum_{k=1}^\ell \sum_{i, j = 1}^{n_k} \frac 1 {n_k} e_{ijk} \tensor \tau(e_{jik})
=
\sum_{k=1}^\ell \sum_{i, j = 1}^{n_k} \frac 1 {n_k} e_{ijk} \tensor e_{ijk}
= e.
$$
\end{proof}

\begin{corollary}\label{I}
Let $A$ and $B$ be finite-dimensional $C^*$-algebras. Let $\varphi\: A \to B$ be a map that is completely positive and completely contractive. Then, $\varphi$ is a homomorphism iff $(\varphi \tensor \id)(\delta)$ is a projection.
\end{corollary}

\begin{proof}
Without loss of generality, $A = M_{n_1}(\CC) \oplus \cdots \oplus M_{n_\ell}(\CC)$ and $B = M_m(\CC)$ for some positive integers $n_1, \ldots, n_\ell$, and $m$. By Theorem~\ref{C}, $\varphi$ is a homomorphism iff $(\varphi \tensor \id)(e)$ is a projection. By Lemma~\ref{H}, $(\varphi \tensor \id)(e) = (\varphi \tensor \id)((\id \tensor \tau)(\delta)) = (\id \tensor \tau)((\varphi \tensor \id)(\delta))$. Thus, $(\varphi \tensor \id)(e)$ is a projection iff $(\varphi \tensor \id)(\delta)$ is a projection. Combining these two equivalences, we obtain the claimed equivalence.
\end{proof}

\begin{remark}
Let $\varphi\: A \to B$ be a unital completely positive map. Then, the conclusion of Corollary~\ref{I} holds because $\varphi$ is completely contractive \cite{EffrosRuan}*{Corollary~5.1.2}.
\end{remark}

\begin{remark}
Let $A = C(\{1, \ldots, n\}) \iso \CC^n$ and $B = C(\{1, \ldots, m\}) \iso \CC^m$, and let $\varphi\: A \to B$ be a unital completely positive map. The states on $A$ and $B$ correspond to the probability distributions on $\{1, \ldots, n\}$ and $\{1, \ldots, m\}$, respectively. Precomposition by $\varphi$ maps states on $B$ to states on $A$, i.e., probability distributions on $\{1, \ldots, m\}$ to probability distributions on $\{1, \ldots, n\}$. Thus, $\varphi$ is essentially a channel $\{1, \ldots, m\} \to\{1, \ldots, n\}$ \cite{Shannon}.

Regarding $\varphi$ as a channel, we may formulate a number of conditions that express the determinism, i.e., the noiselessness, of this channel: First, the channel is deterministic iff $\varphi$ is a homomorphism. Second, the channel is deterministic iff its stochastic matrix $(\varphi \tensor \id)(\delta)$ is a projection. Third, the channel is deterministic iff the entropy of an input probability distribution is never smaller than the entropy of an output probability distribution.

Corollary~\ref{I} establishes the equivalence of the first and second of these conditions in the noncommutative setting. Theorem~\ref{O} establishes the equivalence of the first and third.
\end{remark}

\section{Schr\"{o}dinger picture}

\begin{definition}
Let $A$ and $B$ be finite-dimensional $C^*$-algebras, and let $\varphi \: A \to B$ be a completely positive map.
 We define $\varphi^\dagger\: B \to A$ and $\varphi^\ddagger\: B \to A$ by the following properties:
\begin{enumerate}
\item $\tr(\varphi(a)b) = \tr(a\varphi^\dagger(b))$,
\item $\trr(\varphi(a) b) = \trr(a \varphi^\ddagger(b))$.
\end{enumerate}
\end{definition}
The maps $\varphi^\dagger$ and $\varphi^\ddagger$ can be constructed as the adjoints of $\varphi$ for the inner products $\<a_1|a_2\> = \tr(a_1^*a_2)$ and $(a_1|a_2) = \trr(a_1^* a_2)$, respectively. It is well known that $\varphi^\dagger$ is completely positive. This fact can be viewed as a consequence of Arveson's extension theorem \cite{Arveson} via the Kraus decomposition of $\varphi$. Since the two adjoints are evidently related by $\varphi^\ddagger(b) = \varphi^\dagger(b\zeta_B)\zeta_A^{-1}$, we find that $\varphi^\ddagger$ is also completely positive. It is routine to show that $(\varphi \circ \psi)^\ddagger = \psi^\ddagger \circ \varphi^\ddagger$, that $(\varphi \oplus \psi)^\ddagger = \varphi^\ddagger \oplus \psi^\ddagger$, that $(\varphi \tensor \psi)^\ddagger = \varphi^\ddagger \tensor \psi^\ddagger$, that $\varphi^{\ddagger\ddagger}= \varphi$, and that $\id^\ddagger = \id$.

\begin{lemma}\label{J}
Let $A$ and $B$ be finite-dimensional $C^*$-algebras, and let $\delta_A \in A \tensor A^{op}$ and $\delta_B \in B \tensor B^{op}$ be their diagonal projections. Let $\varphi\: A \to B$ be a completely positive map. Then,
$
(\varphi^\ddagger \tensor \id)(\delta_B) = (\id \tensor \varphi)(\delta_A).
$
\end{lemma}

\begin{proof}
Appealing to Lemma~\ref{G}, we calculate that for all $a \in A$ and $b \in B$,
\begin{align*}
\trr((\varphi^\ddagger \tensor \id)(\delta_B)(a \tensor b))
& =
\trr(\delta_B(\varphi(a) \tensor b))
=
\trr(\varphi(a) b)
=
\trr(a \varphi^\ddagger(b))
\\ & =
\trr(\delta_A(a \tensor \varphi^\ddagger(b)))
=
\trr((\id \tensor \varphi)(\delta_A)(a \tensor b)).
\end{align*}
Thus, $((\varphi^\ddagger \tensor \id)(\delta_B)| a \tensor b) = ((\id \tensor \varphi)(\delta_A)| a \tensor b)$ for all $a \in A$ and $b \in B$.
Since operators of the form $a \tensor b$ span $A \tensor B^{op}$, we conclude that $(\varphi^\ddagger \tensor \id)(\delta_B) = (\id \tensor \varphi)(\delta_A)$.
\end{proof}

A \emph{state} on $A$ is a positive map $\mu\: A \to \CC$ such that $\mu(1) = 1$. It is well known that any state is completely positive \cite{EffrosRuan}*{Lemma~5.1.4}. When $A \iso M_n(\CC)$, each state $\mu\: M_n(\CC) \to \CC$ is commonly associated with its density matrix $d = \mu^\dagger(1) \in M_n(\CC)$:
$$
\mu(a) = \tr( \mu(a)1) = \tr(a \mu^\dagger(1)).
$$
This defines a one-to-one correspondence between states $\mu\: M_n(\CC) \to \CC$ and density matrices $d \in M_n(\CC)$, which generalizes easily to arbitrary finite-dimensional $C^*$-algebras $A$.

\begin{definition}\label{K}
Let $A$ be a finite-dimensional $C^*$-algebra. In this context,
\begin{enumerate}
\item a \emph{density operator} in $A$ is a positive operator $d\in A$ such that $\tr(d) = 1$;
\item an \emph{adjusted density operator} in $A$ is a positive operator $\tilde d \in A$ such that $\trr(\tilde d) = 1$.
\end{enumerate}
\end{definition}

\begin{proposition}
Let $A$ be a finite-dimensional $C^*$-algebra.
\begin{enumerate}
\item States $\mu\: A \to \CC$ are in bijection with density operators $d \in A$:
$$
d = \mu^\dagger(1), \qquad \quad \mu(a) = \tr(ad).
$$
\item States $\mu\: A \to \CC$ are in bijection with adjusted density matrices $\tilde d \in A$:
$$
\tilde d = \mu^\ddagger(1), \qquad \quad \mu(a) = \trr(a \tilde d).
$$
\end{enumerate}
Combining these one-to-one correspondences, we have $d = \tilde d \zeta_A$.
\end{proposition}

\begin{proof}
This proof is entirely routine.
\end{proof}

\begin{definition}
Let $A$ be a finite-dimensional $C^*$-algebra, and let $p \in A$ be a projection.
\begin{enumerate}
\item The \emph{uniform state} on $A$ is the state $ a \mapsto \trr(a)/\dim A$.
\item The \emph{uniform state} on $p$ is the state $a \mapsto \trr(ap)/\trr(p)$.
\end{enumerate}
\end{definition}

\begin{theorem}\label{L}
Let $A$ and $B$ be finite-dimensional $C^*$-algebras, and let $\delta_A \in A \tensor A^{op}$ and $\delta_B \in B \tensor B^{op}$ be their diagonal projections. Let $\mu\: B \tensor B^{op} \to \CC$ be the uniform state on $\delta_B$.
Let $d$ and $\tilde d$ be the density operator and the adjusted density operator of the state
$$\mu \circ (\varphi \tensor \id)\: A \tensor B^{op} \to \CC,$$
respectively. For each unital completely positive map $\varphi\: A \to B$, the following are equivalent:
\begin{enumerate}
\item $\varphi$ is a homomorphism,
\item $(\varphi \tensor \id)(\delta_A)$ is a projection,
\item $(\varphi^\ddagger \tensor \id)(\delta_B)$ is a projection,
\item $(\dim B) \tilde d$ is a projection,
\item $(\dim B) d (\zeta_A \tensor \zeta_B)^{-1}$ is a projection.
\end{enumerate}
\end{theorem}

\begin{proof}
The equivalence $(1) \Leftrightarrow (2)$ is Corollary~\ref{I}. For the equivalence $(2) \Leftrightarrow (3)$, we appeal to Lemma~\ref{J}, calculating that
\begin{align*}&
\sigma((\varphi^\ddagger \tensor \id)(\delta_B)) = \sigma((\id \tensor \varphi)(\delta_A)) = (\varphi \tensor \id)(\sigma(\delta_A)) = (\varphi \tensor \id)(\delta_{A^{op}}),
\end{align*}
where $\sigma$ denotes the appropriate tensor-symmetry $*$-isomorphism throughout. It follows that $(\varphi^\ddagger \tensor \id)(\delta_B)$ is a projection in $A \tensor B^{op}$ iff $(\varphi \tensor \id)(\delta_{A^{op}})$ is a projection in $B^{op} \tensor A$. As an immediate consequence of Definition~\ref{E}, we have that $\delta_{A^{op}}$ and $\delta_{A}$ are equal as vectors, and hence, $(\varphi \tensor \id)(\delta_{A^{op}})$ and $(\varphi \tensor \id)(\delta_{A})$ are equal as vectors. Thus, the former is a projection in $B^{op} \tensor A$ iff the latter is a projection $B \tensor A^{op} = (B^{op} \tensor A)^{op}$. Overall, we have proved the equivalence $(2) \Leftrightarrow (3)$.

For the equivalence $(3) \Leftrightarrow (4)$, we appeal to Lemma \ref{G}, calculating that
\begin{align*}
\tilde d & = ( \mu \circ (\phi \tensor \id))^\ddagger (1) = (\phi \tensor \id)^\ddagger(\mu^\ddagger(1)) = (\phi^\ddagger \tensor \id)(\mu^\ddagger(1))
= 
(\phi^\ddagger \tensor \id)(\trr(\delta_B)^{-1} \delta_B)
\\ & =
\trr(\delta_B)^{-1} (\phi^\ddagger \tensor \id)(\delta_B)
=
\trr(\delta_B(1 \tensor 1))^{-1} (\phi^\ddagger \tensor \id)(\delta_B)
=
\trr(1)^{-1} (\phi^\ddagger \tensor \id)(\delta_B)
\\ & = 
(\dim B)^{-1} (\phi^\ddagger \tensor \id)(\delta_B).
\end{align*}
For the equivalence $(4) \Leftrightarrow (5)$, we simply observe that $\zeta_{A \tensor B^{op}} = \zeta_A \tensor \zeta_{B^{op}} = \zeta_A \tensor \zeta_B$.
\end{proof}

\begin{remark}
The mapping $\mu \mapsto \mu \circ (\varphi \tensor \id)$ is essentially the Choi-Jamio{\l}kowski isomorphism, which is between unital completely positive maps $A \to B$ and certain states on $A \tensor B^{op}$ or, dually, between trace-preserving completely positive maps $B \to A$ and certain density operators in $A \tensor B^{op}$. Of course, Theorem~\ref{L} itself suggests that it is more natural to use the adjusted trace for this duality. Thus, the Choi-Jamio{\l}kowski isomorphism is equivalently between completely positive maps $\varphi\:B \to A$ that preserve the adjusted trace and certain adjusted density operators $\tilde d \in A \tensor B^{op}$.
\end{remark}

\section{Adjusted entropy}

\begin{definition}\label{M}
Let $A$ be a finite-dimensional $C^*$-algebra. Let $\mu\: A \to \CC$ be a state, and let $d = \mu^\dagger(1)$ and $\tilde d = \mu^\ddagger(1)$ be its density operator and its adjusted density operator, respectively. We define
\begin{enumerate}
\item $\displaystyle S(\mu) = -\tr(d \log d)$;
\item $\displaystyle\tilde S (\mu) = -\trr(\tilde d \log \tilde d)$.
\end{enumerate}
We say that $S(\mu)$ is the \emph{entropy} of $\mu$ and that $\tilde S(\mu)$ is the \emph{adjusted entropy} of $\mu$.
\end{definition}

If $A = M_n(\CC)$, then $S(\mu)$ is the von Neumann entropy of $\mu$, and the adjusted entropy of $\mu$ is larger:
\begin{align*}
\tilde S (\mu) & = - \trr (\tilde d \log \tilde d) = - \mu(\log \tilde d) = - \mu(\log(n^{-1}d)) = -\mu(\log d) + \log n \\ & = - \tr(d \log d) + \log n = S(\mu) + \log n.
\end{align*}
This suggests that a quantum system has intrinsic entropy that depends on its dimension.

Definition~\ref{M} requires a standard gloss because $\log d$ is undefined when $d$ is not invertible. We regard $d \log d$ as a notation for $f(d)$, where $f\: [0,\infty) \to \RR$ is the continuous function defined by $f(t) = t \log t$ for $t \in (0,1]$ and $f(0) = 0$.
Some computations will appear to be unnecessarily complicated because they treat this point with some care.

\begin{lemma}
Let $A$ and $B$ be finite-dimensional $C^*$-algebras, and let $\delta_A \in A \tensor A^{op}$ and $\delta_B \in B \tensor B^{op}$ be their diagonal projections. Let $\varphi\: A \to B$ be a unital $*$-homomorphism. Then, $(\varphi \tensor \varphi)(\delta_A) \geq \delta_B$.
\end{lemma}

\begin{proof}
The projection $\delta_A$ is equivalently the largest projection in $A \tensor A^{op}$ that is orthogonal to $\bigvee_p p \tensor (1-p)$, where $p$ ranges over the projections in $A$. In other words, $\delta_A = 1 - \bigvee_p p \tensor (1-p)$. We apply this equation to reason as follows:
\begin{align*}&
\delta_B (q \tensor (1-q)) = 0 \; \text{ for all projections }q \in B
\\ &\IM
\delta_B((\varphi \tensor \varphi)(p \tensor (1-p))) = \delta_B(\varphi(p) \tensor (1-\varphi(p))) = 0 \; \text{ for all projections } p \in A
\\ & \IM
\delta_B\left((\varphi \tensor \varphi)\left(\bigvee_p p \tensor (1-p)\right)\right) = \delta_B\left(\bigvee_p(\varphi \tensor \varphi)( p \tensor (1-p))\right) = 0
\\ & \IM
\delta_B(1 - (\varphi \tensor \varphi)(\delta_A)) = \delta_B (\varphi \tensor \varphi)(1 - \delta_A) = 0
\IM \delta_B \leq (\varphi \tensor \varphi)(\delta_A).
\end{align*}
\end{proof}

\begin{proposition}\label{N}
Let $A$ and $B$ be finite-dimensional $C^*$-algebras, and let $\varphi\: A \to B$ be a unital $*$-homomorphism. Then, $\varphi(\varphi^\ddagger(b)) \geq b$ for all positive $b \in B$.
\end{proposition}

\begin{proof}
Let $b \in B$ be a positive operator. Let $\mu\: B \to \CC$ be a state, and let $\tilde d \in B$ be its adjusted density operator. We apply Lemma~\ref{G} to compute that
\begin{align*}
\mu(\varphi(\varphi^\ddagger(b)))
& =
\trr(\tilde d \varphi(\varphi^\ddagger(b)))
=
\trr(\varphi^\ddagger(\tilde d)\varphi^\ddagger(b))
=
\trr(\delta_A(\varphi^\ddagger(\tilde d) \tensor \varphi^\ddagger(b)))
\\ & =
\trr((\varphi \tensor \varphi)(\delta_A)(\tilde d \tensor b))
\geq
\trr(\delta_B(\tilde d \tensor b))
=
\trr(\tilde d b)
=
\mu(b).
\end{align*}
We find that $\varphi(\varphi^\ddagger(b))$ is a positive operator in $B$ such that $\mu(\varphi(\varphi^\ddagger(b))) \geq \mu(b)$ for all states $\mu\: B \to \CC$. Therefore, $\varphi(\varphi^\ddagger(b)) \geq b$.
\end{proof}

\begin{theorem}\label{O}
Let $A$ and $B$ be finite-dimensional $C^*$-algebras, and let $\varphi\:A \to B$ be a unital completely positive map. The following are equivalent:
\begin{enumerate}
\item $\varphi$ is a homomorphism,
\item all finite-dimensional $C^*$-algebras $C$ and all states $\mu\: B \tensor C \to \CC$ satisfy
$$
\tilde S(\mu \circ (\varphi \tensor \id)) \leq \tilde S(\mu),
$$
\item the uniform state $\mu\: B \tensor B^{op} \to \CC$ on the diagonal projection $\delta_B \in B \tensor B^{op}$ satisfies $$\tilde S(\mu \circ (\varphi \tensor \id)) = \log(\dim B).$$
\end{enumerate}
\end{theorem}

\begin{proof}
We prove the implication $(1) \Rightarrow (2)$. Assume that $\varphi$ is a homomorphism. Let $C$ be a finite-dimensional $C^*$-algebra, and let $\mu\: B \tensor C \to \CC$ be a state. The adjusted density operator of $\mu \circ (\varphi \tensor \id)$ is $(\mu \circ (\varphi \tensor \id))^\ddagger(1) = (\varphi \tensor \id)^\ddagger(\mu^\ddagger(1))$. We calculate that 
\begin{align*}
\tilde S(\mu \circ (\varphi \tensor \id))
& =
- \trr((\varphi \tensor \id)^\ddagger(\mu^\ddagger(1)) \log (\varphi \tensor \id)^\ddagger(\mu^\ddagger(1)))
\\ & =
- \lim_{\epsilon \to 0^+} \trr((\varphi \tensor \id)^\ddagger(\mu^\ddagger(1)) \log ((\varphi \tensor \id)^\ddagger(\mu^\ddagger(1)) + \epsilon 1))
\\ & =
- \lim_{\epsilon \to 0^+} \trr(\mu^\ddagger(1) (\varphi \tensor \id)(\log ((\varphi \tensor \id)^\ddagger(\mu^\ddagger(1)) + \epsilon 1)))
\\ & =
- \lim_{\epsilon \to 0^+} \trr(\mu^\ddagger(1) \log ((\varphi \tensor \id)((\varphi \tensor \id)^\ddagger(\mu^\ddagger(1))) + \epsilon 1))
\\ & \leq
- \lim_{\epsilon \to 0^+} \trr(\mu^\ddagger(1) \log (\mu^\ddagger(1) + \epsilon 1))
= 
-  \trr(\mu^\ddagger(1) \log \mu^\ddagger(1)) = \tilde S(\mu).
\end{align*}
At the inequality step, we reason that $(\varphi \tensor \id)((\varphi \tensor \id)^\ddagger(\mu^\ddagger(1))) \geq \mu^\ddagger(1)$ by Proposition~\ref{N} and, hence, that $\log ((\varphi \tensor \id)((\varphi \tensor \id)^\ddagger(\mu^\ddagger(1))) + \epsilon 1) \geq \log (\mu^\ddagger(1) + \epsilon 1)$ because the logarithm is operator monotone.

We prove the implication $(2) \Rightarrow (3)$. Assume that $\tilde S(\mu \circ (\varphi \tensor \id)) \leq \tilde S(\mu)$ for all finite-dimensional $C^*$-algebras $C$ and all states $\mu\: B \tensor C \to \CC$. Let $C = B^{op}$. Let $\mu\: B \tensor B^{op} \to \CC$ be the uniform state on $\delta_B$. Its adjusted density matrix is $\mu^\ddagger(1) = (\dim B)^{-1} \delta_B$ because $\trr(\delta_B) = \dim B$ by Lemma~\ref{G}. We compute $\tilde S(\mu)$ and $\tilde S (\mu \circ (\varphi \tensor \id))$:

\begin{align*}
\tilde S(\mu) & = - \trr(\mu^\ddagger(1) \log \mu^\ddagger(1)) = - \trr((\dim B)^{-1}\delta_B \log ((\dim B)^{-1} \delta_B))
\\ & =
- \trr((\dim B)^{-1}\delta_B \log \delta_B ) + \trr((\dim B)^{-1}\delta_B \log (\dim B) 1)
\\ & = - \trr((\dim B)^{-1}0) + \trr(\mu^\ddagger(1) \log (\dim B)1)
=
0 + \mu(\log(\dim B)1) = \log(\dim B);
\end{align*}

\begin{align*}
\tilde S (\mu \circ (\varphi \tensor \id))
& =
- \trr((\varphi^\ddagger \tensor \id)(\mu^\ddagger(1)) \log (\varphi^\ddagger \tensor \id)(\mu^\ddagger(1)))
\\ & =
- \trr((\varphi^\ddagger \tensor \id)((\dim B)^{-1} \delta_B) \log (\varphi^\ddagger \tensor \id)((\dim B)^{-1} \delta_B))
\\ & = 
- \trr((\dim B)^{-1}(\varphi^\ddagger \tensor \id)(\delta_B) \log (\varphi^\ddagger \tensor \id)(\delta_B))
\\ & \qquad +
\trr((\varphi^\ddagger \tensor \id)((\dim B)^{-1}\delta_B) \log (\dim B)1)
\\ & =
- (\dim B)^{-1}\trr((\varphi^\ddagger \tensor \id)(\delta_B) \log (\varphi^\ddagger \tensor \id)(\delta_B))
\\ & \qquad +
\log (\dim B) \trr((\varphi^\ddagger \tensor \id)(\mu^\ddagger(1)))
\\ & =
- (\dim B)^{-1}\trr((\varphi^\ddagger \tensor \id)(\delta_B) \log (\varphi^\ddagger \tensor \id)(\delta_B)) + \log(\dim B).
\end{align*}

\quad

\noindent Thus, the inequality $\tilde S(\mu \circ (\varphi \tensor \id)) \leq \tilde S(\mu)$ implies that $\trr((\varphi^\ddagger \tensor \id)(\delta_B) \log (\varphi^\ddagger \tensor \id)(\delta_B)) \geq 0$. However, $(\varphi^\ddagger \tensor \id)(\delta_B) \log (\varphi^\ddagger \tensor \id)(\delta_B)$ is a negative operator because $\|(\varphi^\ddagger \tensor \id)(\delta_B)\| = \|(\id \tensor \varphi)(\delta_A)\| \leq 1$ by Lemma~\ref{J}. Therefore, $\trr((\varphi^\ddagger \tensor \id)(\delta_B) \log (\varphi^\ddagger \tensor \id)(\delta_B)) = 0$, and hence, $\tilde S(\mu \circ (\varphi \tensor \id)) = \log(\dim B)$.

We prove the implication $(3) \Rightarrow (1)$. Let $\mu$ be the uniform state on $\delta_B$, and assume that $\tilde S(\mu \circ (\varphi \tensor \id)) = \log(\dim B)$. We calculate, as we did before, that
$$
\tilde S(\mu \circ (\varphi \tensor \id)) = - (\dim B)^{-1}\trr((\varphi^\ddagger \tensor \id)(\delta_B) \log (\varphi^\ddagger \tensor \id)(\delta_B)) + \log(\dim B),
$$
and hence, $\trr((\varphi^\ddagger \tensor \id)(\delta_B) \log (\varphi^\ddagger \tensor \id)(\delta_B)) = 0$. We infer that $(\varphi^\ddagger \tensor \id)(\delta_B) \log (\varphi^\ddagger \tensor \id)(\delta_B)$ is equal to zero, because it is a negative operator and $\trr\: A \tensor B^{op} \to \CC$ is a faithful trace. Thus, $(\varphi^\ddagger \tensor \id)(\delta_B)$ has spectrum in $\{0,1\}$; therefore, it is a projection. We conclude that $\varphi$ is a homomorphism by Theorem~\ref{L}.
\end{proof}

\begin{proposition}\label{P}
Let $A = M_{n_1}(\CC) \oplus \cdots \oplus M_{n_\ell}(\CC)$. Let $\mu\: A \to \CC$ be a state, and let $d = d_{1} \oplus \cdots \oplus d_{\ell}$ be its density matrix. Then,
$$
S(\mu) = - \sum_{k = 1}^\ell \tr(d_k \log d_k), \qquad  \tilde S(\mu) = - \sum_{k = 1}^\ell \tr(d_k \log d_k) + \sum_{k=1}^\ell \tr(d_k)\log(n_k),
$$
$$
\tilde S(\mu) = S(\mu) + \mu(\log \zeta_A).
$$
\end{proposition}

\begin{proof}
We calculate that

\begin{align*}
-S(\mu) & = \tr(d \log d) = \tr (d_1 \log d_1 \oplus \cdots \oplus d_\ell \log d_\ell)
\\ & = \tr(d_1 \log d_1) + \cdots + \tr(d_\ell \log d_\ell),
\end{align*}
\begin{align*}
- \tilde S(\mu) & = \trr(\tilde d \log \tilde d) = \trr (\zeta_A^{-1} d \log (\zeta_A^{-1} d))
=
\trr(\zeta_A^{-1} d \log d) - \trr(\zeta_A^{-1} d \log \zeta_A)
\\ & = \tr(d \log d) - \tr ( d \log \zeta_A)
= S(\mu) + \mu(\zeta_A),
\end{align*}
\begin{align*}
- \tilde S(\mu) & = \tr(d \log d) - \tr ( d \log \zeta_A)
\\ & = S(\mu) - \tr((\log n_1) d_1 \oplus \ldots \oplus (\log n_\ell) d_\ell)
\\ & = S(\mu) - \tr((\log n_1) d_1) - \ldots - \tr((\log n_\ell) d_\ell)
\\ & = \tr(d_1 \log d_1) + \cdots + \tr(d_\ell \log d_\ell) - \tr(d_1)\log(n_1) - \cdots - \tr(d_\ell)\log(n_\ell).
\end{align*}
\end{proof}

Proposition~\ref{P} implies that for all finite-dimensional $C^*$-algabras $A$ and states $\mu\: A \to \CC$, we have that
$
\tilde S(\mu) = S(\mu) + \mu(\log \zeta_A).
$
We use this equation for the following corollary.

\begin{corollary}\label{Q}
Let $A$ and $B$ be finite-dimensional $C^*$-algebras, and let $\varphi\:A \to B$ be a unital completely positive map. The following are equivalent:
\begin{enumerate}
\item $\varphi$ is a homomorphism,
\item for all positive integers $k$ and all states $\mu\: B \tensor M_k(\CC) \to \CC$,
$$\tilde S(\mu \circ (\varphi \tensor \id)) \leq \tilde S(\mu).$$
\end{enumerate}
\end{corollary}

\begin{proof}
The implication $(1) \Rightarrow (2)$ is an immediate consequence of Theorem~\ref{O}. We prove the implication $(2) \Rightarrow (1)$. Assume $(2)$. Let $C$ be a finite-dimensional $C^*$-algebra, and let $\mu\: B \tensor C \to \CC$ be a state. Without loss of generality, $C$ is a unital $*$-subalgebra of $M_k(\CC)$ for $k = \tr(1)$, and for all $c \in C$, we may evaluate $\tr(c)$ either in $C$ or in $M_k(\CC)$. Let $d \in B \tensor C$ be the density operator of $\mu$; then, let $\mu'\: B \tensor M_{k}(\CC) \to \CC$ be defined by $\mu'(x) = \tr(dx)$. Evidently, $\mu'$ is an extension of $\mu$, and $d$ is the density operator of $\mu'$ as well as of $\mu$. Similarly,
$(\varphi^\dagger \tensor \id)(d) \in A \tensor C$ is the density operator of $\mu' \circ (\varphi \tensor \id)$ as well as of $\mu \circ (\varphi \tensor \id)$.

We observe that
$ S(\mu') = \tr(d \log d) = S(\mu)$. Similarly, $S(\mu' \circ (\varphi \tensor \id)) = S(\mu \circ (\varphi \tensor \id))$. We also observe that
\begin{align*}
(\varphi \tensor \id)(\log(1 \tensor \zeta_C)) & = - (\varphi \tensor \id)(\log(1 \tensor \zeta_C^{-1}))
=
(\varphi \tensor \id) \left(\sum_{n=1}^\infty \frac 1 n 1 \tensor (1- \zeta_C^{-1})^n\right)
\\ & =
\sum_{n=1}^\infty \frac 1 n 1 \tensor (1- \zeta_C^{-1})^n
=
- \log(1 \tensor \zeta_C^{-1})
=
\log(1 \tensor \zeta_C),
\end{align*}
because $\varphi$ is a unital map and $\zeta_C^{-1}$ has spectrum in $(0,1]$. This calculation assumes that our logarithm is the natural logarithm, but this special case clearly implies the general case.

We now calculate that:
\begin{align*}
& \tilde S(\mu' \circ (\varphi \tensor \id))
\\ & =
S(\mu' \circ (\varphi \tensor \id)) + \tr((\varphi^\dagger \tensor \id)(d) \log(\zeta_A \tensor \zeta_C))
\\ & = 
S(\mu \circ (\varphi \tensor \id)) +  \tr((\varphi^\dagger \tensor \id)(d) \log(\zeta_A \tensor 1)) + \tr((\varphi^\dagger \tensor \id)(d)  \log(1 \tensor \zeta_C))
\\ & = 
S(\mu \circ (\varphi \tensor \id)) +  \tr((\varphi^\dagger \tensor \id)(d) \log(\zeta_A \tensor 1)) + \tr(d (\varphi \tensor \id) (\log(1 \tensor \zeta_C)))
\\ & = 
S(\mu \circ (\varphi \tensor \id)) +  \tr((\varphi^\dagger \tensor \id)(d) \log(\zeta_A \tensor 1)) + \tr(d \log(1 \tensor \zeta_C))
\\ & =
S(\mu \circ (\varphi \tensor \id)) +  \tr((\varphi^\dagger \tensor \id)(d) \log(\zeta_A \tensor 1)) + \tr(d \log(1 \tensor \zeta_C))
\\ & \hspace{22em} + \tr((\varphi^\dagger \tensor \id)(d) \log(1\tensor m1)) - \log m
\\ & =
S(\mu \circ (\varphi \tensor \id)) +  \tr((\varphi^\dagger \tensor \id)(d) \log(\zeta_A \tensor m 1)) + \tr(d \log(1 \tensor \zeta_C)) - \log m
\\ & =
\tilde S(\mu \circ (\varphi \tensor \id)) + \tr(d \log(1 \tensor \zeta_C)) - \log m
\end{align*}
Similarly, $\tilde S(\mu') = \tilde S(\mu) + \tr(d \log(1 \tensor \zeta_C)) - \log m$. Thus, we find that $\tilde S(\mu \circ (\varphi \tensor \id))  \leq \tilde S(\mu)$ for all finite-dimensional $C^*$-algebras and all states $\mu\:B \tensor C \to \CC$. Therefore, $\varphi$ is a homomorphism by Theorem~\ref{O}.
\end{proof}

\begin{remark}
We have considered the adjusted entropy $\tilde S$ in the setting of finite-dimensional $C^*$-algebras or, equivalently, finite-dimensional von Neumann algebras. The adjusted entropy $\tilde S$ can be similarly defined in the setting of \emph{hereditarily atomic von Neumann algebras}, which are von Neumann algebras of the form $\bigoplus_{k \in I} M_{n_k}(\CC)$, where $(n_k \suchthat k \in I)$ is any indexed family of positive integers. The notation here refers to the $\ell^\infty$-direct sum of von Neumann algebras. Such von Neumann algebras can be viewed as a quantum generalization of sets \cite{quantumsets}. Of course, the entropy of a normal state on a hereditarily atomic von Neumann algebra may be infinite, even in the commutative case. However, there is always an obvious separating family of normal states with finite entropy, in contrast to Remark~\ref{B}.

\end{remark}

\section*{Acknowledgements}

I thank Peter Selinger for discussing this problem with me. I thank Piotr So{\l}tan for introducing me to the functional that I have called the adjusted trace and for suggesting a reference. I thank an anonymous referee for suggesting the connection to relative entropy.

\end{document}